\documentclass{article}

\RequirePackage[OT1]{fontenc}
\RequirePackage[
amsthm,amsmath
]{imsart}

\usepackage{graphicx}
\usepackage{latexsym,amsmath}
\usepackage{amsmath,amsthm,amscd}
\usepackage{amsfonts}
\usepackage[psamsfonts]{amssymb}
\usepackage{enumerate}

\usepackage{url}

\startlocaldefs
\numberwithin{equation}{section}

\theoremstyle{plain} 
\newtheorem{theorem}{Theorem}[section]
\newtheorem{corollary}[theorem]{Corollary}

\newtheorem{lemma}[theorem]{Lemma}
\newtheorem{proposition}[theorem]{Proposition}

\theoremstyle{definition} 

\theoremstyle{definition} 

\newtheorem*{ex*}{Example}
\theoremstyle{remark} 

\theoremstyle{remark} 
\newtheorem{remark}[theorem]{Remark}
\newtheorem*{remark*}{Remark}
\numberwithin{equation}{section}
 
\newcommand{\beqa}{\begin{eqnarray}}
\newcommand{\eeqa}{\end{eqnarray}}
 
\newcommand{\bseq}{\begin{subequations}}
 
\newcommand{\eseq}{\end{subequations}}

\newcommand{\dd}{\partial}

\newcommand{\p}{\mathcal{P}}

\renewcommand{\dd}{{\,\operatorname{d}}}

\newcommand{\al}{\alpha}
\newcommand{\ga}{\gamma}

\newcommand{\si}{\sigma}
\newcommand{\ka}{\kappa}
\newcommand{\la}{\lambda}
\newcommand{\de}{\delta}
\newcommand{\be}{\beta}

\newcommand{\vpi}{\varphi}

\renewcommand{\Psi}{\overline{\Phi}}

\newcommand{\PP}{\operatorname{\mathsf{P}}} 
\newcommand{\E}{\operatorname{\mathsf{E}}}

\newcommand{\R}{\mathbb{R}}

\newcommand{\X}{\mathbf{X}}
\newcommand{\Y}{\mathbf{Y}}
\newcommand{\x}{\mathbf{x}}
\newcommand{\y}{\mathbf{y}}
\newcommand{\e}{\mathbf{e}}
\newcommand{\s}{\mathbf{s}}
\renewcommand{\v}{\mathbf{v}}
\renewcommand{\u}{\mathbf{u}}

\newcommand{\vp}{\varepsilon}

\newcommand{\tp}{{\tilde{p}}}
\newcommand{\tq}{{\tilde{q}}}

\renewcommand{\le}{\leqslant}
\renewcommand{\ge}{\geqslant}

\endlocaldefs

\begin{document}

\begin{frontmatter}

\title{Exponential deficiency of convolutions of densities}
\runtitle{Deficiency of convolutions of densities}

%

\begin{aug}
\author{\fnms{Iosif} \snm{Pinelis}\thanksref{t2}\ead[label=e1]{ipinelis@mtu.edu}}
  \thankstext{t2}{Supported by NSF grant DMS-0805946}
\runauthor{Iosif Pinelis}


\address{Department of Mathematical Sciences\\
Michigan Technological University\\
Houghton, Michigan 49931, USA\\
E-mail: \printead[ipinelis@mtu.edu]{e1}}
\end{aug}

\begin{abstract}
If a probability density $p(\x)$ ($\x\in\R^k$) is bounded and $R(t):=\int\e^{\langle\x,t\u\rangle}\,p(\x)\dd\x<\infty$ for some linear functional $\u$ and all $t\in(0,1)$, then, for each $t\in(0,1)$ and all large enough $n$, the $n$-fold convolution of the $t$-tilted density $\tilde p_t(\x):=e^{\langle\x,t\u\rangle}p(\x)/R(t)$ is bounded. 
This is a corollary of a general, ``non-i.i.d.'' result, which is also shown to enjoy a certain optimality property. 
Such results are useful for saddle-point approximations.  
\end{abstract}

\begin{keyword}[class=AMS]
\kwd[Primary ]{60E05}
\kwd{60E10}
\kwd[; secondary ]{60F10}
\kwd{62E20}
\kwd{60E15}
\end{keyword}

\begin{keyword}
\kwd{probability density}
\kwd{saddle-point approximation}
\kwd{sums of independent random variables/vectors}
\kwd{convolution}
\kwd{exponential integrability}
\kwd{boundedness}
\kwd{tilting}
\kwd{exponential families}
\end{keyword}

\end{frontmatter}


{\small\tableofcontents} 



\theoremstyle{plain} 
\numberwithin{equation}{section}

\eject

\section{Introduction}\label{intro}

Let $\X$ be a random vector in $\R^k$
such that 
\begin{equation}\label{eq:exp}
	M:=\E e^{\la\,\e\X}<\infty 
\end{equation}
for some unit vector $\e\in\R^k$ and some $\la\in(0,\infty)$; here the juxtaposition $\e\x$ denotes the Euclidean scalar product of vectors $\e$ and $\x$ in $\R^k$. 
By Chebyshev's inequality, the exponential integrability condition \eqref{eq:exp} implies the tail estimate 
\begin{equation}\label{eq:cheb}
	\PP(\e\X\ge x)\le M\,e^{-\la x}\quad\text{for all $x\in\R$. }
\end{equation}
Vice versa, for any given $\la_0\in(0,\infty]$ one has the following: if \eqref{eq:cheb} holds for each $\la\in[0,\la_0)$ and some  $M=M(\la)\in(0,\infty)$, then $\E e^{\la\e\X}<\infty$ for each $\la\in[0,\la_0)$.

Suppose also that (the distribution of) $\X$ has a density $p$ (relative to the Lebesgue measure) such that, for some $\mu\in[0,\la)$ and some $C\in(0,\infty)$, 
\begin{equation}\label{eq:mu-bound}
	p(\x)\le C\,e^{-\mu\,\e\x}\quad\text{for all $\x\in\R^k$.}
\end{equation}
Note also that, if $\mu=0$, then condition \eqref{eq:mu-bound} simply means that the density $p$ is bounded. 

If $p$ is varying regularly enough in an appropriate sense then, given the condition \eqref{eq:exp}, one will have \eqref{eq:mu-bound} for $\mu=\la$; 
that is, one will have an exact ``local'' counterpart to the ``integral'' upper bound \eqref{eq:cheb}. 
The difference  
\begin{equation*}
	\vp:=\la-\mu
\end{equation*}
\big(between the largest possible $\la$ and $\mu$ 
for which \eqref{eq:exp} and \eqref{eq:mu-bound} will still hold\big)  
may therefore be referred to as the (exponential) ``deficiency'' of the density $p$, which is a measure of its irregularity. 
 
The main result of this paper implies that the deficiency decreases fast under convolution: starting with condition \eqref{eq:mu-bound} for $p$ with $\mu=\la-\vp$, one has this condition for the $n$-fold convolution $p^{*n}$ (in place of $p$) with $\mu=\la-\vp/n$; that is, for the $n$-fold convolution, the deficiency is $n$ times as small as the original one. More generally, it is proved that, for any probability densities $p_1,\dots,p_n$ on $\R^k$ satisfying the exponential integrability condition with the same $\la$ and with respective deficiencies $\vp_1,\dots,\vp_n$, the deficiency of the convolution $p_1*\dots*p_n$ is no greater than $\vp^\sharp/n$, where $\vp^\sharp$ stands for the harmonic mean of the original deficiencies $\vp_1,\dots,\vp_n$. Moreover, it is shown that this bound, $\vp^\sharp/n$, cannot be improved. 
  
\section{Statements of the results}\label{results}
Let $\X_1,\dots,\X_n$ be any independent random vectors in $\R^k$, with densities $p_1,\dots,p_n$.

Assume the following conditions:
\begin{gather}
M_i:=\E e^{\la\,\e\X_i}=\int_{\R^k}e^{\la\,\e\x}p_i(\x)\dd\x<\infty \label{eq:exp_i}\\
	\intertext{and}
		p_i(\x)\le C_i\,e^{-\mu_i\,\e\x} \label{eq:mu-bound_i}
\end{gather} 
for some $C_i$'s in $(0,\infty)$, some $\mu_i$'s in $[0,\la)$, all $i\in\{1,\dots,n\}$, and all $\x\in\R^k$. 
Consider the convolution 
\begin{equation}\label{eq:p}
	p^{(n)}:=p_1*\dots*p_n,
\end{equation}
which is the density of the sum $\X_1+\dots+\X_n$.

\begin{theorem}\label{th:}
There exists a finite constant 
$K_n$, which depends only on the numbers $n$, $\la$, $\mu_i$, $M_i$, and $C_i$,  
such that 
\begin{equation}\label{eq:th}
	p^{(n)}(\x)\le K_n\,e^{-(\la-\vp^{(n)})\,\e\x}\quad\text{for all $\x\in\R^k$, }
\end{equation}
where 
\begin{equation}\label{eq:vp}
	\vp^{(n)}:=\frac1{\frac1{\vp_1}+\dots+\frac1{\vp_n}}\quad\text{and}\quad
	\vp_i:=\la-\mu_i>0.
\end{equation}
\end{theorem} 

The necessary proofs will be given in Section~\ref{proofs}. 

Note that $\vp^{(n)}=\vp^\sharp/n$, where $\vp^\sharp$ denotes the harmonic mean of $\vp_1,\dots,\vp_n$. One may also note that $\vp^{(n)}<\min(\vp_1,\dots,\vp_n)$. 

It turns out that the coefficient $\la-\vp^{(n)}$ in the exponent in the bound \eqref{eq:th} is the best possible:  

\begin{proposition}\label{prop:best}
For any natural $k$ and $n$, any $\la\in(0,\infty)$, and any $\mu_i$'s in $[0,\la)$, the estimate \eqref{eq:th} will fail to hold if 
the number $\vp^{(n)}$ given by \eqref{eq:vp} is replaced by any smaller number. 
\end{proposition}


From Theorem~\ref{th:}, one immediately obtains the particular ``i.i.d.'' case: 

\begin{corollary}\label{cor:iid}
If conditions \eqref{eq:exp} and \eqref{eq:mu-bound} hold, then 
for each natural $n$ there exists a constant 
$K_n$, which depends only on the numbers $n$, $\la$, $\mu$, $M$, and $C$, 
such that
\begin{equation}\label{eq:iid}
	p^{*n}(\x)\le K_n\,e^{-(\la-\vp/n)\,\e\x}\quad\text{for all $\x\in\R^k$, }
\end{equation}
where $\vp:=\la-\mu$. 
\end{corollary} 

It follows from Proposition~\ref{prop:best} that the coefficient $\la-\vp/n$ in the exponent in the bound \eqref{eq:iid} is the best possible. 

In turn, Corollary~\ref{cor:iid} yields 

\begin{corollary}\label{cor:tilt}
If conditions \eqref{eq:exp} and \eqref{eq:mu-bound} hold, then for each $t\in(0,\la)$ there exists a natural number $n_t$ such that for all natural $n\ge n_t$ the $n$-fold convolution $\tilde p_t^{*n}$ of the $t$-tilted density 
\begin{equation}\label{eq:p tilt}
	\tilde p_t(\x):=\frac{e^{t\,\e\x}p(\x)}{\E e^{t\,\e\X}}\quad (\x\in\R^k)  
\end{equation}
is bounded. 
\end{corollary}

In fact, in Corollary~\ref{cor:tilt} one may take $n_t=\lceil\frac{\la-\mu}{\la-t}\rceil$. 

Corollary~\ref{cor:tilt} can be rewritten as 

\begin{corollary}\label{cor:tilt Fourier}
If conditions \eqref{eq:exp} and \eqref{eq:mu-bound} hold, then for each $t\in(0,\la)$ there exists some $\ga_t\in(0,\infty)$ such that for all $\ga\ge\ga_t$ 
\begin{equation}\label{eq:f tilt}
\int_{\R^k}|\tilde f_t(\s)|^\ga\dd\s<\infty,  
\end{equation}
where $\tilde f_t(\s):=\int_{\R^k}e^{i\,\s\x}\,\tilde p_t(\x)\dd\x$, the characteristic function of the $t$-tilted density $\tilde p_t$; here, of course, $i$ stands for the imaginary unit. 
\end{corollary}

\begin{remark}\label{rem:grouping} In applications, one may of course assume the ``grouping'': $\X_j=\Y_{m_{j-1}+1}+\dots+\Y_{m_j}$ 
for $j=1,\dots,n$, where $0=m_0<m_1<\dots$ and the $\Y$'s are independent random vectors, whose distributions may themselves not have a density. Then the densities $p_1,\dots,p_n$ as in Theorem~\ref{th:} will be the densities of the convolutions of the distributions of the corresponding $\Y$'s. 
\end{remark} 

\section{Discussion}\label{discuss}

The condition of the boundedness of the $n$-fold convolution $\tilde p_t^{*n}$ of the tilted density $\tilde p_t$ or, equivalently, the condition \eqref{eq:f tilt} of the absolute integrability of the corresponding ``tilted'' characteristic function is needed to derive saddle-point approximations. Surveys of literature on such approximations are given e.g.\ in \cite{daniels87,reid}; for more recent work see e.g.\ \cite{shao-saddle,shao-saddle-review}.

In the context of saddle-point approximations, the tilting is sometimes described as imbedding the original density $p$ into the exponential family \eqref{eq:p tilt}. The condition of the boundedness of $\tilde p_t^{*m}$ for \emph{all} relevant values of the tilting parameter $t$ and all large enough $m$ appears to be usually imposed outright; see e.g.\ Barndorff-Nielsen and 
Cox~\cite[page~298, condition c]{barn-n-cox}; Lugannani and Rice \cite[page~481, condition (ii)]{lug-rice} impose an even stronger condition, requiring (for $k=1$) that $|\tilde f_t(s)|=O((1+|s|)^{-\ga})$ for some $\ga>0$ and all $s\in\R$. 

On the other hand, Corollaries~\ref{cor:tilt} and \ref{cor:tilt Fourier} together with Remark~\ref{rem:grouping} show that one need \emph{a priori} require the boundedness of $\tilde p_t^{*m}$ only for $t=0$ and some natural  $m$, that is, only for some convolution $p^{*m}$ of the original, un-tilted density $p$; then $\tilde p_t^{*m}$ will necessarily be bounded for all $t$ in the interval $[0,\la)$ and all large enough $m$. 

The considerations presented above in this section constituted the original motivation for the present work. 
%
The proof of Proposition~\ref{prop:best} (given in the next section) shows that probability densities with the deficiencies most resistant to convolution are mixtures of infinitely many mutually (almost) singular densities, spaced regularly enough (see Fig.\ \ref{fig:1} on page~\pageref{fig:1}). 
Such ``exponentially deficient'' distributions can be contrasted with the well-studied classes of regualrly behaving distributions with nearly exponential tails; see e.g.\ \cite{embr-goldie,kluppelberg,pin85}.  

\section{Proofs}\label{proofs}

\begin{proof}[Proof of Theorem~\ref{th:}]
To begin, note that for $n=1$ the inequality \eqref{eq:th} with $K_1:=C_1$ is the same as \eqref{eq:mu-bound_i}. 
Next, a trivial remark is that \eqref{eq:exp_i} implies $\E e^{\la\,\e(\X_1+\dots+\X_{n-1})}=M_1\cdots M_{n-1}<\infty$. 
Note also that \eqref{eq:vp} can rewritten in an additive form, as 
\begin{equation*}
	\frac1{\vp^{(n)}}=\frac1{\vp_1}+\dots+\frac1{\vp_n}. 
\end{equation*}

So, by induction, it suffices to prove Theorem~\ref{th:} for $n=2$. For such a case, let us simplify the notation by writing $p$ and $q$ instead of $p_1$ and $p_2$, $M$ and $N$ instead of $M_1$ and $M_2$, $C$ and $D$ instead of $C_1$ and $C_2$, $\mu$ and $\nu$ instead of $\mu_1$ and $\mu_2$, and $\vp$ and $\de$ instead of $\vp_1$ and $\vp_2$. 

Next, without loss of generality, $\e=(1,0,\dots,0)\in\R^k$. Then, identifying any vector $\x\in\R^k$ with the corresponding pair $(x,\y)\in\R\times\R^{k-1}$, one has $\e\x=x$, so that \eqref{eq:mu-bound_i} can in this case be rewritten as 
\begin{equation}\label{eq:mu-rewr}
	p(x,\y)\le C\,e^{-\mu x} \quad\text{and}\quad
	q(x,\y)\le D\,e^{-\nu x}
\end{equation}
for all $(x,\y)\in\R\times\R^{k-1}$. 
Also, conditions \eqref{eq:exp_i} imply 
\begin{align}
	\int_x^\infty \dd u\,\tp(u)\le M\,e^{-\la x} \quad\text{and}\quad
	\int_x^\infty \dd u\,\tq(u)\le N\,e^{-\la x} \label{eq:exp-rewr}
\end{align} 
for all $x\in\R$, where 
\begin{equation*}
	\tp(u):=\int_{\R^k}\dd\v\,p(u,\v) \quad\text{and}\quad
	\tq(u):=\int_{\R^k}\dd\v\,q(u,\v)
\end{equation*}
for all $u\in\R$, the densities of the random variables $\e\X_1$ and $\e\X_2$, respectively. 

Fix now any $(x,\y)\in\R\times\R^{k-1}$. 
Take, for a moment, any $\al\in(0,1)$ and let $\be:=1-\al$. 
Then
\begin{align}
	(p*q)(x,\y)&=\int_\R\dd u\,\int_{\R^{k-1}}\dd\v\,p(x-u,\y-\v)\,q(u,\v)\le D\,I_1+C\,I_2\label{eq:I1+I2}
\end{align}
by \eqref{eq:mu-rewr}, 
where  
\begin{align}
	I_1&:=\int_{-\infty}^{\al x}\dd u\,\int_{\R^{k-1}}\dd\v\,p(x-u,\y-\v)\,e^{-\nu u} 
	=\int_{-\infty}^{\al x}\dd u\,\tp(x-u)\,e^{-\nu u}, \notag\\
	I_2&:=\int_{\al x}^\infty\dd u\,\int_{\R^{k-1}}\dd\v\,p(x-u,\y-\v)q(u,\v) 
	=\int_{-\infty}^{\be x}\dd z\,\tq(x-z)\,e^{-\mu z}. \notag 
\end{align}
Next, in view of \eqref{eq:exp-rewr}, 
\begin{align}
	I_1&=\int_{-\infty}^{\al x}\dd u\,\tp(x-u)\,\int_u^\infty\nu\dd z\,e^{-\nu z} \\
	&=\int_{-\infty}^{\al x}\nu\dd z\,e^{-\nu z}\int_{-\infty}^z\dd u\,\tp(x-u)
	+ \int_{\al x}^\infty\nu\dd z\,e^{-\nu z}\int_{-\infty}^{\al x}\dd u\,\tp(x-u) \notag\\
	&=\int_{-\infty}^{\al x}\nu\dd z\,e^{-\nu z}\int_{x-z}^\infty\dd w\,\tp(w)
	+ \int_{\al x}^\infty\nu\dd z\,e^{-\nu z}\int_{(1-\al)x}^\infty\dd w\,\tp(w) \notag\\
	&\le\int_{-\infty}^{\al x}\nu\dd z\,e^{-\nu z}\,M\,e^{-\la(x-z)}
	+ \int_{\al x}^\infty\nu\dd z\,e^{-\nu z}\,M\,e^{-\la(1-\al)x} \notag\\
	&=M\,\frac\la{\la-\nu}\,e^{-\big(\la-(\la-\nu)\al\big)x}. \label{eq:I11} 
\end{align}
Note that this derivation of the upper bound \eqref{eq:I11} on $I_1$ is valid only for $\nu\ne0$. 
However, if $\nu=0$, then 
$$I_1=\int_{-\infty}^{\al x}\dd u\,\tp(x-u)=\int_{(1-\al)x}^\infty\dd z\,\tp(z)\le M\,e^{-\la(1-\al)x},$$
so that the bound \eqref{eq:I11} on $I_1$ holds for $\nu=0$ as well. 
Recall now that $\vp=\la-\mu$ and $\de=\la-\nu$, and choose
$\al:=\frac\vp{\vp+\de}$. 
Then \eqref{eq:I11} can be rewritten as 
\begin{equation}\label{eq:I1 bound}
	I_1\le M\,\frac\la\de\,e^{-(\la-\vp^{(2)})x}, 
\end{equation}
with $\vp^{(2)}=\frac{\vp\de}{\vp+\de}=\frac1{1/\vp+1/\de}$, in accordance with the definition \eqref{eq:vp} of $\vp^{(n)}$. 
Quite similarly, 
\begin{equation}\label{eq:I2 bound}
	I_2\le N\,\frac\la\vp\,e^{-(\la-\vp^{(2)})x}.  
\end{equation}

Collecting now \eqref{eq:I1+I2}, \eqref{eq:I1 bound}, and \eqref{eq:I2 bound}, one sees that 
\begin{equation}\label{eq:n=2}
	(p*q)(x,\y)\le K_2\,e^{-(\la-\vp^{(2)})x}
\end{equation}
for some constant $K_2$ depending only on $\la$, $\mu$, $M$, $N$, $C$, and $D$, and for all $(x,\y)\in\R\times\R^{k-1}$. 
%
Thus, Theorem~\ref{th:} is proved for $n=2$ and, thereby, for all natural $n$.
\end{proof}

The proof of Proposition~\ref{prop:best} rests on Lemma~\ref{lem:} below. To state the lemma, for any $\la\in(0,\infty)$ and $\vp\in(0,\la]$ introduce the class $\p_{\la,\vp}$ of all probability densities $p$ on $\R$ such that 
\begin{enumerate}[(i)]
	\item $\int_\R e^{\la x}p(x)\dd x<\infty$ and 
	\item $p(x)\ge c\,p_{\la,\vp,\ka,\al}(x)$ for some $c\in(0,\infty)$, $\ka\in(0,\infty)$, $\al\in(\frac12,\infty)$, and all $x\in\R$, where
\begin{align}
	 p_{\la,\vp,\ka,\al}(x)&:=\sum_{j=-\infty}^\infty W_j(x), \label{eq:p_ def}\\
	 W_j(x)&:=W_{j;\la,\vp,\ka,\al}(x):=w_j\,f_{j,\ka e^{-\vp|j|}}(x), \label{eq:W}\\
	 w_j&:=w_{j;\la,\al}:=\frac{e^{-\la|j|}}{(j^2+1)^\al}, \label{eq:aj}\\
	 f_{a,b}(x)&:=\frac1b\,\vpi\Big(\frac{x-a}b\Big),\quad \vpi(u):=\frac1{\sqrt{2\pi}}\,e^{-u^2/2}; \notag
\end{align}
of course, $f_{a,b}$ is the density of the normal distribution with mean $a$ and variance $b^2$. 
\end{enumerate}

\ \hspace{-3pt}\big(One could similarly, and even a little more easily, deal with the ``asymmetric'' version of the class $\p_{\la,\vp}$, having $\sum_{j=-\infty}^\infty$ in \eqref{eq:p_ def} replaced by $\sum_{j=0}^\infty$.\big) 

\begin{lemma}\label{lem:}\ Take any $\la\in(0,\infty)$, $\vp\in(0,\la]$, $\ka\in(0,\infty)$, and $\al\in(\frac12,\infty)$. 
\begin{enumerate}[(I)]
	\item There exists some $c_{\la,\vp,\ka,\al}\in(0,\infty)$ such that 
	$\tilde p_{\la,\vp,\ka,\al}:=\dfrac{p_{\la,\vp,\ka,\al}}{c_{\la,\vp,\ka,\al}}\in\p_{\la,\vp}$. In particular, it follows that $\p_{\la,\vp}\ne\emptyset$. 
	\item There exists some $C=C_{\la,\vp,\ka,\al}\in(0,\infty)$ such that for $p=p_{\la,\vp,\ka,\al}$ and $\mu:=\la-\vp$  
\begin{equation}\label{eq:mu-bound,p}
	p(x)\le C\,e^{-\mu x} \quad\text{for all $x\in\R$.}
\end{equation}
\item For any $p\in\p_{\la,\vp}$ and any $C\in(0,\infty)$, relation \eqref{eq:mu-bound,p} does not hold with any $\mu^\diamond\in(\la-\vp,\infty)$ in place of $\mu$. 
\item In addition to $\vp$, take any $\de\in(0,\la]$.  
Then, for any $p\in\p_{\la,\vp}$ and $q\in\p_{\la,\de}$, one has $p*q\in\p_{\la,\tilde\vp}$, where 
\begin{equation}\label{eq:tvp}
	\tilde\vp:=\frac1{\frac1\vp+\frac1\de}=\frac{\vp\de}{\vp+\de}.
\end{equation}
\end{enumerate}
\end{lemma}

The (symmetric about $0$) probability density $\tilde p_{\la,\vp,\ka,\al}$ as in part (I) of this lemma is illustrated here: 

\begin{figure*}[h]
	\includegraphics[scale=0.75]{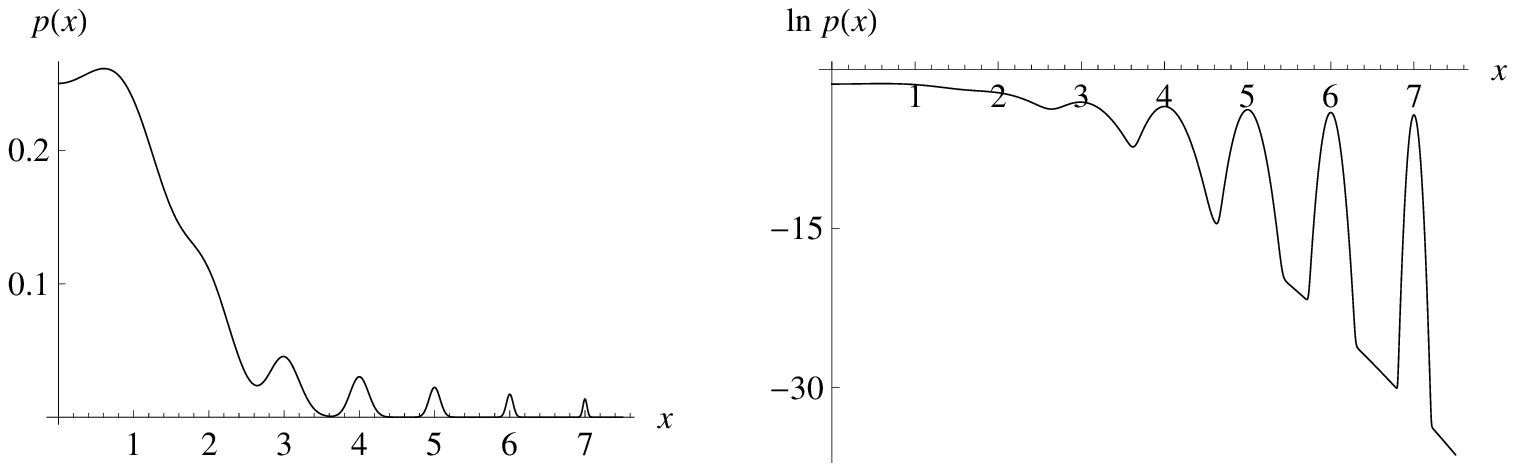}
	\caption{Graphs $\{(x,p(x))\colon0<x<7.5\}$ and $\{(x,\ln p(x))\colon0<x<7.5\}$ for $p=\tilde p_{\la,\vp,\ka,\al}$ with $\la=0.55$, $\vp=0.50$, $\ka=0.9$, and $\al=0.6$
	}
	\label{fig:1}
\end{figure*}

Let us postpone the proof of Lemma~\ref{lem:}, which is somewhat long, and proceed now to the proof of Proposition~\ref{prop:best}.

\begin{proof}[Proof of Proposition~\ref{prop:best}]
Take indeed any natural $k$ and $n$, any $\la\in(0,\infty)$, and any $\mu_1,\dots,\mu_n$ in $[0,\la)$. In accordance with 
\eqref{eq:vp}, let $\vp_i:=\la-\mu_i$, so that $\vp_i\in(0,\la]$ for all $i=1,\dots,n$. 
For each $i=1,\dots,n$, take any density $q_i\in\p_{\la,\vp_i}$ such that
\begin{equation}\label{eq:qi}
	q_i(x)\le C_i\,e^{-\mu_i x}
\end{equation}
for some finite positive real constant $C_i$ and all $x\in\R$; by parts (I) and (II) of Lemma~\ref{lem:}, such $q_i$'s do exist. 

As in the proof of Theorem~\ref{th:}, let $\e=(1,0,\dots,0)\in\R^k$ and identify any vector $\x\in\R^k$ with  $(x,\y)\in\R\times\R^{k-1}$. Then, for each $i=1,\dots,n$, introduce the densities
\begin{equation}\label{eq:p,q,vpi}
	p_i(\x)=p_i(x,\y):=q_i(x)\vpi_{k-1}(\y)
\end{equation}
for all $\x=(x,\y)\in\R\times\R^{k-1}$, where $\vpi_{k-1}(\y):=(2\pi)^{-(k-1)/2}\,e^{-\y\y/2}$ for all $\y\in\R^{k-1}$; 
then   
\begin{equation*}
	\int_{\R^k}e^{\la\,\e\x}p_i(\x)\dd\x=\int_\R e^{\la x}q_i(x)\dd x<\infty, 
\end{equation*}
since $q_i\in\p_{\la,\vp_i}$;  
also, by \eqref{eq:qi},  
\begin{equation*}
	p_i(\x)=q_i(x)\vpi_{k-1}(\y)
	\le(2\pi)^{-(k-1)/2}\,q_i(x)
	\le C_i\,e^{-\mu_i x}
\end{equation*} 
for all $\x=(x,\y)\in\R\times\R^{k-1}$. 
So, conditions \eqref{eq:exp_i} and \eqref{eq:mu-bound_i} hold. 

Next, introduce 
\begin{equation*}
	q^{(n)}:=q_1*\dots*q_n,
\end{equation*}
so that, by \eqref{eq:p} and \eqref{eq:p,q,vpi},
\begin{equation}\label{eq:q,vpi}
	p^{(n)}(\x)=p^{(n)}(x,\y)=q^{(n)}(x)\vpi_{k-1}^{*n}(\y)
\end{equation} 
for all $\x=(x,\y)\in\R\times\R^{k-1}$. 
Moreover, recalling the conditions $q_i\in\p_{\la,\vp_i}$ for $i=1,\dots,n$ and using part (IV) of Lemma~\ref{lem:}, by induction one concludes that $q^{(n)}\in\p_{\la,\vp^{(n)}}$. 

Now, to obtain a contradiction, assume that \eqref{eq:th} holds with some ``deficiency'' $\vp^\diamond$ in place of $\vp^{(n)}$ such that $\vp^\diamond<\vp^{(n)}$. Then, by \eqref{eq:q,vpi}, for $\mu^\diamond:=\la-\vp^\diamond$ 
\begin{equation*}
	q^{(n)}(x)\vpi_{k-1}^{*n}(\mathbf{0})\le K_n\,e^{-\mu^\diamond x}
\end{equation*} 
for some constant $K_n$ and all $x\in\R$. 
But this contradicts part (III) of Lemma~\ref{lem:}, since $\mu^\diamond\in(\la-\vp^{(n)},\infty)$, $q^{(n)}\in\p_{\la,\vp^{(n)}}$, and $\vpi_{k-1}^{*n}(\mathbf{0})=(2\pi n)^{-(k-1)/2}>0$. 
This concludes the proof of Proposition~\ref{prop:best}, except that one still needs to prove Lemma~\ref{lem:}. 
\end{proof}

\begin{proof}[Proof of Lemma~\ref{lem:}]\ 

(I)\quad Obviously, $p_{\la,\vp,\ka,\al}>0$ and 
$c_{\la,\vp,\ka,\al}:=\int_\R p_{\la,\vp,\ka,\al}(x)\dd x=\sum_{j=-\infty}^\infty w_j<\infty$ 
for any $\la\in(0,\infty)$, $\vp\in(0,\la]$, $\ka\in(0,\infty)$, and $\al\in(\frac12,\infty)$. 
So, $\tilde p_{\la,\vp,\ka,\al}$ is a probability density. 
Moreover, 
\begin{equation}
	\int_\R e^{\la x}p_{\la,\vp,\ka,\al}(x)\dd x
	=\sum_{j=-\infty}^\infty w_j\,\exp\big(\la j+\tfrac12\,\la^2\ka^2\,e^{-2\vp|j|}\big)
	\le\sum_{j=-\infty}^\infty \frac{e^{\la^2\ka^2/2}}{(j^2+1)^\al}<\infty. 
\end{equation}
Thus, part (I) of Lemma~\ref{lem:} is verified. 

(II)\quad Note that 
\begin{equation}\label{eq:Aj}
	W_j(x)=\frac1{\ka\sqrt{2\pi}}\,
	\frac{e^{-(\la-\vp)|j|}}{(j^2+1)^\al}
	\,\exp-\frac{(x-j)^2\,e^{2\vp|j|}}{2\ka^2}. 
\end{equation}
Hence and because $\vp\in(0,\la]$, one has 
$
	p_{\la,\vp,\ka,\al}(x)\le C:=\sum_{j=-\infty}^\infty W_j(j)<\infty
$ 
for all $x\in\R$. 
So, $p_{\la,\vp,\ka,\al}(x)\le C\,e^{-\mu x}$ for all $x\in(-\infty,0]$; that is, \eqref{eq:mu-bound,p} holds for $p=p_{\la,\vp,\ka,\al}$, $\mu=\la-\vp$, and all $x\in(-\infty,0]$. 

Take now any $x\in(0,\infty)$. Introduce $j_x:=\lfloor x\rfloor$, so that $0\le j_x\le x<j_x+1$ and for $j\ge j_x$ one has $|j|=j>x-1$. Then, in view of \eqref{eq:Aj}, 
\begin{equation}\label{eq:c1}
	\sum_{j=j_x}^\infty W_j(x)\le\sum_{j=j_x}^\infty W_j(j)
	\le\frac{e^{-(\la-\vp)(x-1)}}{\ka\sqrt{2\pi}}\,
\sum_{j=0}^\infty \frac1{(j^2+1)^\al}=c_1\,e^{-(\la-\vp)x}; 
\end{equation}
in this proof of part (II) of the lemma, let $c_1,c_2,\dots$ denote finite positive constants depending only on 
$\la,\vp,\ka,\al$. 
Next, for $r_x:=\big\lceil\ka\sqrt{2(\la-\vp)x}\,\,\big\rceil$ and $j\in(-\infty,j_x-r_x]$, one has 
$x-j\ge r_x$, whence 
\begin{equation}\label{eq:c2}
	\sum_{j=-\infty}^{j_x-r_x} W_j(x)
	\le
	\sum_{j=-\infty}^{j_x-r_x} W_j(j)\,\exp-\frac{r_x^2}{2\ka^2}
	\le c_2\,\exp-\frac{r_x^2}{2\ka^2}\le c_2\,e^{-(\la-\vp)x}.  
\end{equation}
Further, for $j\in[j_x-r_x+1,j_x-1]$ one has $x-j\ge1$ and $|j|\ge j\ge j_x-r_x+1>x-r_x\ge\frac x2-c_3$, whence 
\begin{multline}\label{eq:c3}
	\sum_{j=j_x-r_x+1}^{j_x-1} W_j(x)
	\le
	\sum_{j=j_x-r_x+1}^{j_x-1} W_j(j)\,\exp-\frac{e^{2\vp(x-r_x)}}{2\ka^2}
	\le c_4\exp-\frac{e^{\vp(x-2c_3)}}{2\ka^2} \\
	\le c_5\,e^{-(\la-\vp)x}.  
\end{multline}
So, by \eqref{eq:p_ def}, \eqref{eq:c1}, \eqref{eq:c2}, and \eqref{eq:c3}, the relation 
\eqref{eq:mu-bound,p} (with $\mu=\la-\vp$) holds for $p=p_{\la,\vp,\ka,\al}$ and all $x\in(0,\infty)$ as well. 
This completes the verification of part (II) of the lemma.

(III)\quad Take any $p\in\p_{\la,\vp}$, so that $p\ge c\,p_{\la,\vp,\ka,\al}$ for some $c\in(0,\infty)$, $\ka\in(0,\infty)$, and $\al\in(\frac12,\infty)$. Then 
\begin{equation*}
	p(j)\ge c\,p_{\la,\vp,\ka,\al}(j)\ge c\,W_j(j)
	=\frac c{\ka\sqrt{2\pi}}\,
	\frac{e^{-(\la-\vp)j}}{(j^2+1)^\al}>C\,e^{-\mu^\diamond j}
\end{equation*}
for any $\mu^\diamond\in(\la-\vp,\infty)$, any $C\in(0,\infty)$, and all large enough natural $j$. 
This proves part (III) of the lemma. 

(IV)\quad Take any $p\in\p_{\la,\vp}$  and $q\in\p_{\la,\de}$, so that $p\ge c\,p_{\la,\vp,\ka,\al}$ and $q\ge \tilde c\,p_{\la,\de,\xi,\be}$ for some $c\in(0,\infty)$, $\ka\in(0,\infty)$, $\al\in(\frac12,\infty)$, $\tilde c\in(0,\infty)$, $\xi\in(0,\infty)$, and $\be\in(\frac12,\infty)$. 

Choose for a moment any $m\in\{0,1,\dots\}$ and let 
\begin{equation}\label{eq:ij}
	i_m:=\Big\lfloor m\frac\de{\vp+\de}\Big\rfloor\quad\text{and}\quad 
	j_m:=\Big\lceil m\frac\vp{\vp+\de}\Big\rceil=m-i, 
\end{equation}
so that $m\frac\de{\vp+\de}-1\le i_m\le m\frac\de{\vp+\de}$ and 
$m\frac\vp{\vp+\de}\le j_m\le m\frac\vp{\vp+\de}+1$. 
Next, introduce 
\begin{gather*}
	\si_m:=\sqrt{\ka^2e^{-2\vp|i_m|}+\xi^2e^{-2\de|j_m|}}=\sqrt{\ka^2e^{-2\vp i_m}+\xi^2e^{-2\de j_m}},\\
	\zeta:=\sqrt{\ka^2+\xi^2e^{-2\de}},\quad
	\tilde\zeta:=\sqrt{\ka^2e^{2\vp}+\xi^2},
\end{gather*}
and observe that 
\begin{gather*}
\frac{\tilde\zeta}{\zeta}\le e^{\vp\vee\de}, \\
\si_m^2\ge\ka^2\,\exp\Big\{-2\vp m\,\frac\de{\vp+\de}\Big\}
+\xi^2\,\exp\Big\{-2\de\,\Big(m\,\frac\vp{\vp+\de}+1\Big)\Big\}
=\zeta^2\,e^{-2\tilde\vp m}, \\
\si_m^2\le\ka^2\,\exp\Big\{-2\vp\,\Big(m\,\frac\de{\vp+\de}-1\Big)\Big\}
+\xi^2\,\exp\Big\{-2\de m\,\frac\vp{\vp+\de}\Big\}
=\tilde\zeta^2\,e^{-2\tilde\vp m},
\end{gather*}
where $\tilde\vp$ is as in \eqref{eq:tvp}. 
Also, recall that here $m\ge0$, $i_m\ge0$, and $j_m\ge0$. 
It follows that for all $x\in\R$ 
\begin{multline*}
	\big(f_{i_m,\ka e^{-\vp|i_m|}}*f_{j_m,\xi e^{-\de|j_m|}}\big)(x)
	=f_{m,\si_m^2}(x)=\frac1{\si_m\sqrt{2\pi}}\,\exp-\frac{(x-m)^2}{2\si_m^2} \\
	\ge \frac{\zeta e^{-\tilde\vp m}}{\si_m}\,f_{m,\zeta\,e^{-\tilde\vp m}}(x)
	\ge \frac{\zeta}{\tilde\zeta}\,f_{m,\zeta\,e^{-\tilde\vp m}}(x)
	\ge e^{-(\vp\vee\de)}\,f_{m,\zeta\,e^{-\tilde\vp|m|}}(x).
\end{multline*}

Quite similarly (or by symmetry), one has 
\begin{equation*}
		f_{i_m,\ka e^{-\vp|i_m|}}*f_{j_m,\xi e^{-\de|j_m|}}
	\ge e^{-(\vp\vee\de)}\,f_{m,\zeta\,e^{-\tilde\vp|m|}}
\end{equation*}
for any $m\in\{-1,-2,\dots\}$, letting now $i_m:=-i_{-m}=\big\lceil m\frac\de{\vp+\de}\big\rceil$ and $j_m:=-j_{-m}=\big\lfloor m\frac\vp{\vp+\de}\big\rfloor$, so that still $i_m+j_m=m$. 

On recalling the conditions $p\ge c\,p_{\la,\vp,\ka,\al}$, $q\ge \tilde c\,p_{\la,\de,\xi,\be}$, \eqref{eq:p_ def}--\eqref{eq:aj}, and \eqref{eq:ij}, 
it follows that 
\begin{multline*}
	p*q\ge c\tilde c
	\sum_{m=-\infty}^\infty w_{i_m;\la,\al}\,w_{j_m;\la,\be}\,f_{i_m,\ka e^{-\vp|i_m|}}*f_{j_m,\xi e^{-\de|j_m|}} \\
	\ge c_1
	\sum_{m=-\infty}^\infty w_{i_m;\la,\al}\,w_{j_m;\la,\be}\,f_{m,\zeta\,e^{-\tilde\vp|m|}} \\
	\ge c_2
	\sum_{m=-\infty}^\infty w_{m;\la,\al+\be}\,f_{m,\zeta\,e^{-\tilde\vp|m|}} 
	=c_2\,p_{\la,\tilde\vp,\zeta,\al+\be},
\end{multline*}
where $c_1$ and $c_2$ are strictly positive constants depending only on $\la,\vp,\de,\ka,\xi,\al,\be$. 

Also, $\int_\R e^{\la x}(p*q)(x)\dd x=\int_\R e^{\la x}p(x)\dd x\,\int_\R e^{\la x}q(x)\dd x<\infty$. 
Thus, it has been shown that $p*q\in\p_{\la,\tilde\vp}$. 
This completes the verification of part (IV). 
The lemma is now completely proved. 
\end{proof}

\begin{proof}[Proof of Corollary~\ref{cor:tilt}]
This follows because 
$\displaystyle{
	\tilde p_t^{*n}(\x)=\frac{e^{t\,\e\x}p^{*n}(\x)}{(\E e^{t\,\e\X})^n}
	}$ for all \break
	$\x\in\R^k$. 
\end{proof}

\begin{proof}[Proof of Corollary~\ref{cor:tilt Fourier}]
Take any $t\in(0,\la)$. Then, by 
Corollary~\ref{cor:tilt}, $\tilde p_t^{*n_t}$ is bounded by some constant $K<\infty$. Then, by the Plancherel isometry (see e.g.\ \cite[Theorem 4.2]{bhat}), for all $\ga\ge2n_t$
\begin{equation*}
	\int_{\R^k}|\tilde f_t(\s)|^\ga\,\dd\s\le\int_{\R^k}|\tilde f_t(\s)|^{2n_t}\,\dd\s
	=(2\pi)^k\int_{\R^k}\tilde p_t^{*n_t}(\x)^2\,\dd\x
	\le (2\pi)^k\,K<\infty. 
\end{equation*}
Vice versa, assume that \eqref{eq:f tilt} holds for all $\ga\ge\ga_t$; then $\tilde p_t^{*n}$ is bounded for all natural $n\ge\ga_t$ by the Fourier inversion formula (see e.g.\ \cite[Theorem 4.1(iv)]{bhat}), 
since the characteristic function of $\tilde p_t^{*n}$ is $\tilde f_t(\s)^n$. 
\end{proof}

\begin{remark}\label{rem:planch}
Weaker results than the one given by Theorem~\ref{th:} or even Corollary~\ref{cor:iid} (but which  still be enough to deduce Corollaries~\ref{cor:tilt} and \ref{cor:tilt Fourier}) can be obtained more simply modulo the Plancherel isometry. Indeed, if conditions \eqref{eq:exp} and \eqref{eq:mu-bound} hold, then  
\begin{equation*}
	\int_{\R^k}|\tilde f_t(\s)|^2\dd\s
	=(2\pi)^k\int_{\R^k}\tilde p_t(\x)^2\dd\x
	=\frac{(2\pi)^k}{(\E e^{t\,\e\X})^2}\,\int_{\R^k}e^{2t\,\e\x}\,p(\x)^2\dd\x
	<\infty
\end{equation*}
for $t=\la-\vp/2$ and $\vp:=\la-\mu$, since $e^{2(\la-\vp/2)\,\e\x}\,p(\x)^2\le C\,e^{\la\,\e\x}\,p(\x)$ for all $\x\in\R$. Also, by the Fourier inversion formula
, again with $t=\la-\vp/2$,  
\begin{equation*}
	\tilde p_t^{*2}(\x)\le(2\pi)^{-k}\int_{\R^k}|\tilde f_t(\s)|^2\dd\s<\infty\quad\text{for all $\x\in\R^k$, } 
\end{equation*}
which yields \eqref{eq:iid} for $n=2$. 
Thus, by induction, one can obtain \eqref{eq:iid} for $n=2^j$, where $j$ is any natural number. 

However, it is unclear whether such an approach, via the Plancherel isometry, could be extended to yield Theorem~\ref{th:} or, at least,  Corollary~\ref{cor:iid} for all natural $n$. Anyway, it might be not worthwhile to exert efforts in such a direction, as the direct probabilistic proof of Theorem~\ref{th:} given above is rather simple already and yet produces the best possible bound on the exponential deficiency. 
\end{remark}

\bibliographystyle{acm}
\bibliography{citat}



\end{document}